\documentclass[11pt,reqno]{amsart} 

\usepackage{amssymb,latexsym}
\usepackage[draft=true]{hyperref}
\usepackage{cite} 
\usepackage{epsfig}
\usepackage[height=190mm,width=130mm]{geometry} 

\theoremstyle{plain}
\newtheorem{theorem}{Theorem}
\newtheorem{lemma}{Lemma}
\newtheorem{corollary}{Corollary}

\theoremstyle{definition}

\theoremstyle{remark}
\newtheorem{remark}{Remark}

\numberwithin{equation}{section} 

\begin{document}
\title[On a system of difference equations... ]{On a system of difference equations of second order solved in a closed from} 

\author{Y. Akrour}
\address{Youssouf Akrour,  \'Ecole Normale Sup\'erieure de Constantine\\ D\'epartement  des Sciences Exactes et Informatiques, Alg\'erie\\ and LMAM Laboratory, University of Mohamed Seddik Ben Yahia, Jijel\\ Algeria.}

\email{youssouf.akrour@gmail.com}

\author{N. Touafek}

\address{Nouressadat Touafek, University of Mohamed Seddik Ben Yahia\\LMAM Laboratory and Department of Mathematics , Jijel\\ Algeria.}
\email{ntouafek@gmail.com}

\author{Y. Halim}
\address{Yacine Halim, University Center of Mila\\ Department of Mathematics and Computer Science and LMAM Laboratory, University of Mohamed Seddik Ben Yahia, Jijel\\ Algeria.}

\email{halyacine@yahoo.fr}

\begin{abstract}
In this work we solve in closed form the  system of difference equations
\begin{equation*}
   x_{n+1}=\dfrac{ay_nx_{n-1}+bx_{n-1}+c}{y_nx_{n-1}},\; y_{n+1}=\dfrac{ax_ny_{n-1}+by_{n-1}+c}{x_ny_{n-1}},\;n=0,1,...,
\end{equation*}
where the initial values $x_{-1}$, $x_0$, $y_{-1}$ and $y_0$ are arbitrary nonzero real numbers and the parameters $a$, $b$ and $c$ are arbitrary real numbers with $c\ne 0$. In particular we represent the solutions of some particular cases of this system in terms of Tribonacci and Padovan numbers and we prove the global stability of the corresponding positive equilibrium points. The result obtained here extend those obtained in some recent papers.
\end{abstract}




\maketitle

\section{Introduction}

We find in the literature many studies that  concern the representation of the solutions of some remarkable linear sequences such as Fibonacci, Lucas, Pell, Jacobsthal, Padovan, and Perrin (see, e.g., \cite{Alladi-Hoggatt:1997}, \cite{Falcon and Plaza:2007}, \cite{Irmak:2013}, \cite{Koshy:2001}, \cite{McCarty:1981,Shannon_Horadam1972,Shannon Anderson and Horadam:2006}, \cite{Yazlik and Taskara:2012}). Solving in closed form non linear difference equations and systems is a subject that highly attract the attention of researchers (see, e.g.,\cite{Elsayed:2015,Elsayed Ibrahim:2015,Elsayed:2014,Halim Rabago:2018,Halim Bayram:2016,matsunaga,stevicejqtde,stevicejqtde2018,Tollu Yazlik Taskara:2013,Tollu Yazlik Taskara:2014(1),Yazlik Tollu Taskara:2013}) and the reference cited therein, where we find very interesting formulas of the solutions. A large range of these formulas are expressed in terms of famous numbers like Fibonacci and Padovan, (see, e.g., \cite{Halim Rabago:2018}, \cite{stevicejqtde}, \cite{Tollu Yazlik Taskara:2013}). For solving in closed form non linear difference equations and systems generally we use some change of variables that transformed nonlinear equations and systems in linear ones. The paper of Stevic \cite{stevic2004} has considerably motivated this line of research.

The difference equation $$x_{n+1}=a+\frac{b}{x_{n-1}}+\frac{c}{x_{n}x_{n-1}}$$ was studied by Azizi in \cite{azizi}. Noting that the same equation was the subject of a very recent paper by Stevic \cite{stevicejqtde2018}.

In \cite{Yazlik Tollu Taskara:2013} the authors studied the system
$$x_{n+1}=\frac{1+x_{n-1}}{y_{n}x_{n-1}},\,y_{n+1}=\frac{1+y_{n-1}}{x_{n}y_{n-1}},$$
Motivated by \cite{Yazlik Tollu Taskara:2013}, Halim et al. in \cite{Halim Rabago:2018}, got the form of the solutions of  the following difference equation
$$x_{n+1}=\frac{a+bx_{n-1}}{x_{n}x_{n-1}},$$ and the system $$x_{n+1}=\frac{a+bx_{n-1}}{y_{n}x_{n-1}},\,y_{n+1}=\frac{a+by_{n-1}}{x_{n}y_{n-1}},$$
Here and motivated by the above mentioned papers we are interested in the following system of difference equations
\begin{equation}\label{system-generalized}
 x_{n+1}=\frac{ay_{n}x_{n-1}+bx_{n-1}+c}{y_{n}x_{n-1}},\,y_{n+1}=\frac{ax_{n}y_{n-1}+by_{n-1}+c}{x_{n}y_{n-1}},\,n=0,1,...,
\end{equation}
where  $x_{-1}, x_0,y_{-1}$ and $y_0$ are arbitrary nonzero real numbers, $a$, $b$ and $c$ are arbitrary real numbers with $c \neq 0$. Clearly our system generalized the equations and systems studied in \cite{azizi}, \cite{Halim Rabago:2018}, \cite{stevicejqtde2018} and  \cite{Yazlik Tollu Taskara:2013}.

\section{The homogenous third order linear difference equation with constant coefficients.}

Consider the homogenous third order linear difference equation

\begin{equation}\label{R(n)}
R_{n+1}=aR_{n}+bR_{n-1}+cR_{n-2},\,n=0,1,...,
\end{equation}
where the initial values $R_{0}, R_{-1}$ and $R_{-2}$  and the constant coefficients $a$, $b$ and $c$ are real numbers with $c\neq 0$. This equation will be of great importance for our study, so we will solve it in a closed form. As it is well known, the solution $\left(R_{n}\right)_{n=-2}^{+\infty}$ of equation \eqref{R(n)} is usually expressed in terms of the roots $\alpha$, $\beta$ and $\gamma$  of the characteristic equation
\begin{equation}
\lambda^3-a\lambda^2-b\lambda-c=0   \label{eq_charac_of_GTS}
\end{equation}
Here we express the solutions of the equation \eqref{R(n)} using terms of the sequence $\left(J_{n}\right)_{n=0}^{+\infty}$ defined by the recurrent relation
\begin{equation}\label{tribonacci_sequence_J(n)}
J_{n+3}=aJ_{n+2}+bJ_{n+1}+cJ_{n},\quad n\in \mathbb{N},
\end{equation}
and the special initial values
\begin{equation}
J_0=0, \quad J_1=1 \; \text{and} \; J_2=a
\end{equation}

Noting that $\left(R_{n}\right)_{n=-2}^{+\infty}$ and $\left(J_{n}\right)_{n=0}^{+\infty}$ have the same characteristic equation. Also  if $a=b=c=1$, then the equation \eqref{tribonacci_sequence_J(n)} is nothing other then the famous Tribonacci sequence $\left(T_{n}\right)_{n=0}^{+\infty}$.

The closed form of the solutions of $\{J_{n}\}_{n=0}^{+\infty}$ and many proprieties of them are well known in the literature, for the interest of the readers and for the purpose of our work,  we show how we can get the formula of the solutions and we give also a result on  the limit $$\lim_{n\rightarrow \infty} \dfrac{J_{n+1}}{J_n}.$$

For the roots  $\alpha$, $\beta$ and $\gamma$ of characteristic equation \eqref{eq_charac_of_GTS}, we have
\begin{equation} \label{relation racines}
\begin{cases}
\alpha + \beta + \gamma=a \\
\alpha \beta+\alpha \gamma+\beta \gamma =-b \\
\alpha \beta \gamma =c.
\end{cases}
\end{equation}
 We have:

\textbf{Case 1: If all roots are equal.} In this case
\begin{equation*}
J_n=\left(  c_{1}+c_{2}n+c_{3}n^2\right) \alpha^n.
\end{equation*}
Now using \eqref{relation racines} and the fact that $J_{0}=0$, $J_{1}=1$ and $J_{2}=a$, we obtain
\begin{equation}  \label{binet_formula1_of_J(n)}
J_n=\left(  \dfrac{n}{2\alpha}+\dfrac{n^2}{2\alpha}\right) \alpha^n
\end{equation}

\textbf{Case 2: If two roots are equal, say $\beta=\gamma$.} In this case
\begin{equation*}
J_n=c_{1} \alpha^n+\left(c_{2}+c_{3}n\right)\beta^n.
\end{equation*}
Now using \eqref{relation racines} and the fact that $J_{0}=0$, $J_{1}=1$ and $J_{2}=a$, we obtain
\begin{equation} \label{binet_formula2_of_J(n)}
J_n=\dfrac{\alpha}{(\beta-\alpha)^2} \alpha^n +\left(  \dfrac{-\alpha}{(\beta-\alpha)^2}+\dfrac{n}{\beta-\alpha}\right)\beta^n,
\end{equation}

\textbf{Case 3: If the roots are all different.} In this case
\begin{equation*}
J_n=c_{1} \alpha^n+c_{2}\beta^n+c_{3}\gamma^{n}.
\end{equation*}
Again, using \eqref{relation racines} and the fact that $J_{0}=0$, $J_{1}=1$ and $J_{2}=a$, we obtain
\begin{equation} \label{binet_formula3_of_J(n)}
J_n=\dfrac{\alpha}{(\gamma-\alpha)(\beta-\alpha)}  \alpha^n +\dfrac{-\beta}{(\gamma-\beta)(\beta-\alpha)} \beta^n +\dfrac{\gamma}{(\gamma-\alpha)(\gamma-\beta)} \gamma^n
\end{equation}

In this case we can get two roots of \eqref{eq_charac_of_GTS}  complex conjugates say $\gamma=\overline{\beta}$ and the third one real and the formula of $J_{n}$ will be
\begin{equation} \label{binet_formula4_of_J(n)}
J_n=\dfrac{\alpha}{(\overline{\beta}-\alpha)(\beta-\alpha)}  \alpha^n +\dfrac{-\beta}{(\overline{\beta}-\beta)(\beta-\alpha)} \beta^n +\dfrac{\overline{\beta}}{(\overline{\beta}-\alpha)(\overline{\beta}-\beta)} {\overline{\beta}}^n
\end{equation}

Consider the following linear third order difference equation
\begin{equation} \label{third-order_equation_S(n)}
S_{n+1}=-aS_{n}+bS_{n-1}-cS_{n-2},\,n=0,1,...,
\end{equation}
the constant coefficients $a$, $b$ and $c$ and the initial values $S_{0}, S_{-1}$ and $S_{-2}$ are real numbers. As for the equation \eqref{R(n)}, we will express the solutions of \eqref{third-order_equation_S(n)} using terms of \eqref{tribonacci_sequence_J(n)}. To do this let us consider the difference equation
\begin{equation}\label{tribonacci_sequence_j(n)}
j_{n+3}=-aj_{n+2}+bj_{n+1}-cj_{n},\quad n\in \mathbb{N},
\end{equation}
and the special initial values
\begin{equation}
j_0=0, \quad j_1=1 \; \text{and} \; j_2=-a
\end{equation}

The characteristic equation of \eqref{third-order_equation_S(n)} and \eqref{tribonacci_sequence_j(n)} is
   \begin{equation}
\lambda^3+a\lambda^2-b\lambda+c=0.   \label{eq_charac_of_GTS1}
\end{equation}
Clearly the roots of  \eqref{eq_charac_of_GTS1} are $-\alpha$, $-\beta$ and $-\gamma$. Now following the same procedure in solving $\{J(n)\}$, we get that
$$j(n)=(-1)^{n+1}J(n).$$

\begin{lemma} \label{limits_of_J(n)}
Let $\alpha$, $\beta$ and $\gamma$ be the roots of \eqref{eq_charac_of_GTS}, assume that $\alpha$ is a real root with $\max(\vert\alpha\vert;\vert\beta\vert; \vert\gamma\vert)=\vert\alpha\vert$. Then,
$$\lim_{n\rightarrow \infty} \dfrac{J_{n+1}}{J_n}=\alpha. $$
\end{lemma}

\begin{proof}

If $\alpha$, $\beta$ and $\gamma$ are real and distinct then,
\begin{align*}
\lim_{n \rightarrow \infty} \dfrac{J_{n+1}}{J_n} &=\lim_{n \rightarrow \infty} \dfrac{\dfrac{\alpha}{(\gamma-\alpha)(\beta-\alpha)}  \alpha^{n+1} +\dfrac{-\beta}{(\gamma-\beta)(\beta-\alpha)} \beta^{n+1} +\dfrac{\gamma}{(\gamma-\alpha)(\gamma-\beta)} \gamma^{n+1}}{\dfrac{\alpha}{(\gamma-\alpha)(\beta-\alpha)}  \alpha^n +\dfrac{-\beta}{(\gamma-\beta)(\beta-\alpha)} \beta^n +\dfrac{\gamma}{(\gamma-\alpha)(\gamma-\beta)} \gamma^n} \\
&=\lim_{n \rightarrow \infty} \dfrac{\alpha^{n+1}}{\alpha^n}  \dfrac{\dfrac{\alpha}{(\gamma-\alpha)(\beta-\alpha)} \dfrac{\alpha^{n+1}}{\alpha^{n+1}} +\dfrac{-\beta}{(\gamma-\beta)(\beta-\alpha)} \dfrac{\beta^{n+1}}{\alpha^{n+1}} +\dfrac{\gamma}{(\gamma-\alpha)(\gamma-\beta)} \dfrac{\gamma^{n+1}}{\alpha^{n+1}}}{\dfrac{\alpha}{(\gamma-\alpha)(\beta-\alpha)} \dfrac{\alpha^n}{\alpha^n} +\dfrac{-\beta}{(\gamma-\beta)(\beta-\alpha)} \dfrac{\beta^n}{\alpha^n} +\dfrac{\gamma}{(\gamma-\alpha)(\gamma-\beta)} \dfrac{\gamma^n}{\alpha^n}}\\
&=\lim_{n \rightarrow \infty} \alpha  \dfrac{\dfrac{\alpha}{(\gamma-\alpha)(\beta-\alpha)} +\dfrac{-\beta}{(\gamma-\beta)(\beta-\alpha)} \left( \dfrac{\beta}{\alpha}\right)^{n+1} +\dfrac{\gamma}{(\gamma-\alpha)(\gamma-\beta)} \left( \dfrac{\gamma}{\alpha}\right)^{n+1} }{\dfrac{\alpha}{(\gamma-\alpha)(\beta-\alpha)} +\dfrac{-\beta}{(\gamma-\beta)(\beta-\alpha)} \left(\dfrac{\beta}{\alpha}\right)^n  +\dfrac{\gamma}{(\gamma-\alpha)(\gamma-\beta)} \left( \dfrac{\gamma}{\alpha}\right)^n}\\
&= \alpha
\end{align*}
The proof of the other cases  of the roots, that is when $\alpha=\beta=\gamma$ or $\beta$, $\gamma$ are complex conjugate,   is similar to the first one and will be omitted.
\end{proof}

\begin{remark}
  If $\alpha$ is real root and $\beta$, $\gamma$ are complex conjugate  with $$\max(\vert\alpha\vert;\vert\beta\vert; \vert\overline{\beta}\vert)=\vert\beta\vert= \vert\overline{\beta}\vert,$$ then
$\lim\limits_{n \rightarrow \infty} \dfrac{J_{n+1}}{J_n} $ doesn't exist.
\end{remark}

In the following result, we solve in a closed form the equations \eqref{R(n)} and  \eqref{third-order_equation_S(n)} in terms of the sequence $\left(J_{n}\right)_{n=0}^{+\infty}$. The obtained formula will be very useful to obtain the formula of the solutions of system  \eqref{system-generalized}.
\begin{lemma} \label{lemma_R(n),S(n)_by_J(n)}
We have for all $n \in \mathbb{N}_0$,
\begin{equation}\label{R(n)_by_J(n)}
R_n= cJ_n R_{-2}+(J_{n+2}-aJ_{n+1})R_{-1}+J_{n+1}R_0
\end{equation}
\begin{equation}\label{S(n)_by_J(n)}
S_n= (-1)^{n}\left[cJ_n S_{-2}+(-J_{n+2}+aJ_{n+1})S_{-1}+J_{n+1}S_0\right].
\end{equation}
\end{lemma}
\begin{proof}

Assume that $\alpha$, $ \beta$ and $\gamma$ are the distinct roots of the characteristic  equation  \eqref{eq_charac_of_GTS}, so
\begin{equation*}
R_{n}=c'_{1}\alpha^{n}+c'_{2}\beta^{n}+c'_{3}\gamma^{n},\,n=-2,-1,0,....
\end{equation*}
Using the initial values $R_{0}, R_{-1}$ and $R_{-2}$, we get
\begin{equation} \label{roots_and__c(i)_relations_of_R(n)}
\begin{cases}
\dfrac{1}{\alpha^2}c'_1 + \dfrac{1}{\beta^2}c'_2 + \dfrac{1}{\gamma^2}c'_3 &=R_{-2} \\
\dfrac{1}{\alpha}c'_1 + \dfrac{1}{\beta}c'_2 + \dfrac{1}{\gamma}c'_3 &=R_{-1} \\
c'_1+c'_2+c'_3 &=R_0
\end{cases}
\end{equation}
after some  calculations we get
\begin{eqnarray*}
c'_1 & = & \dfrac{\alpha^2 \beta \gamma}{(\gamma-\alpha)(\beta-\alpha)}R_{-2} -\dfrac{(\gamma+\beta)\alpha^2}{(\gamma-\alpha)(\beta-\alpha)}R_{-1}+\dfrac{\alpha^2}{(\gamma-\alpha)(\beta-\alpha)}R_0 \\
c'_2 & = & -\dfrac{\alpha \beta^2 \gamma}{(\gamma-\beta)(\beta-\alpha)}R_{-2}+\dfrac{(\alpha+\gamma)\beta^2}{(\gamma-\beta)(\beta-\alpha)}R_{-1}-\dfrac{\beta^2}{(\gamma-\beta)(\beta-\alpha)}R_0\\
c'_3 & = & \dfrac{\alpha \beta \gamma^2}{(\gamma-\alpha)(\gamma-\beta)}R_{-2}-\dfrac{(\alpha+\beta)\gamma^2}{(\gamma-\alpha)(\gamma-\beta)}R_{-1}+\dfrac{\gamma^2}{(\gamma-\alpha)(\gamma-\beta)}R_0
\end{eqnarray*}
that is,
\begin{eqnarray*}
R_n & = & \left(\dfrac{\alpha^2 \beta \gamma}{(\gamma-\alpha)(\beta-\alpha)}\alpha^n  -\dfrac{\alpha \beta^2 \gamma}{(\gamma-\beta)(\beta-\alpha)}\beta^n  +\dfrac{\alpha \beta \gamma^2}{(\gamma-\alpha)(\gamma-\beta)}\gamma^n \right) R_{-2} \\
 && +\: \left(-\dfrac{(\gamma+\beta)\alpha^2}{(\gamma-\alpha)(\beta-\alpha)}\alpha^n +\dfrac{(\alpha+\gamma)\beta^2}{(\gamma-\beta)(\beta-\alpha)}\beta^n -\dfrac{(\alpha+\beta)\gamma^2}{(\gamma-\alpha)(\gamma-\beta)}\gamma^n \right) R_{-1} \\
&& + \: \left(\dfrac{\alpha^2}{(\gamma-\alpha)(\beta-\alpha)}\alpha^n -\dfrac{\beta^2}{(\gamma-\beta)(\beta-\alpha)}\beta^n +\dfrac{\gamma^2}{(\gamma-\alpha)(\gamma-\beta)}\gamma^n \right) R_0 \\
\end{eqnarray*}
\begin{equation*}
R_n= cJ_n R_{-2}+(J_{n+2}-aJ_{n+1})R_{-1}+J_{n+1}R_0.
\end{equation*}
The proof of the other cases is similar and will be omitted.

Let $A:=-a$ and $B:=b$, $C:=-c$, then equation \eqref{third-order_equation_S(n)} takes the form of \eqref{R(n)} and the equation \eqref{tribonacci_sequence_j(n)} takes the form of \eqref{tribonacci_sequence_J(n)}. Then analogous to the formula of \eqref{R(n)} we obtain  \begin{equation*}
S_n= Cj_n S_{-2}+(j_{n+2}-Aj_{n+1})S_{-1}+j_{n+1}S_0.
\end{equation*}
Using the fact that  $j(n)=(-1)^{n+1}J(n)$, $A=-a$ and $C:=-c$ we get
\begin{equation*}
S_n= (-1)^{n}\left(cJ_n S_{-2}-(J_{n+2}-aJ_{n+1})S_{-1}+J_{n+1}S_0\right).
\end{equation*}
\end{proof}
\section{Closed form of well defined solutions of system \eqref{system-generalized}}
In this section, we solve through an analytical approach the system  \eqref{system-generalized} with $c\neq 0$ in a closed form. By a well defined solutions of system \eqref{system-generalized}, we mean a solution that satisfies $x_{n}y_{n}\neq 0,\,n=-1,0,\cdots$. Clearly if we choose the initial values and the parameters $a$, $b$ and $c$ positif, then every solution of \eqref{system-generalized} will be well defined.

The following result give an explicit formula for well defined solutions of the system \eqref{system-generalized}.

\begin{theorem} \label{solution_form_of_x(n)_and_y(n)}
Let $\{x_n, y_n\}_{n\geq -1}$ be a well defined solution of
\eqref{system-generalized}. Then, for $n=0,1,\ldots,$ we have

\begin{equation*}
x_{2n+1}=\dfrac{cJ_{2n+1}+(J_{2n+3}-aJ_{2n+2})x_{-1}+J_{2n+2}x_{-1}y_{0}}{ cJ_{2n}+(J_{2n+2}-aJ_{2n+1})x_{-1}+J_{2n+1}x_{-1}y_{0}},
\end{equation*}

\begin{equation*}
x_{2n+2}=\dfrac{cJ_{2n+2}+(J_{2n+4}-aJ_{2n+3})y_{-1}+J_{2n+3}x_{0}y_{-1}}{ cJ_{2n+1}+(J_{2n+3}-aJ_{2n+2})y_{-1}+J_{2n+2}x_{0}y_{-1}},
\end{equation*}

\begin{equation*}
y_{2n+1}=\dfrac{ cJ_{2n+1}+(J_{2n+3}-aJ_{2n+2})y_{-1}+J_{2n+2}x_{0}y_{-1}}{cJ_{2n}+(J_{2n+2}-aJ_{2n+1})y_{-1}+J_{2n+1}x_{0}y_{-1}},
\end{equation*}

\begin{equation*}
y_{2n+2}=\dfrac{ cJ_{2n+2}+(J_{2n+4}-aJ_{2n+3})x_{-1}+J_{2n+3}x_{-1}y_{0}}{cJ_{2n+1}+(J_{2n+3}-aJ_{2n+2})
x_{-1}+J_{2n+2}x_{-1}y_{0}}
\end{equation*}

where the initial conditions $x_{-1},x_{0}, y_{-1}$ and $y_{0}\in
\left(\mathbb{R}-\left\{0\right\}\right)-F$, with $F$ is the Forbidden set of system
\eqref{system-generalized} given by
\begin{equation*}
F=\bigcup_{n=0}^\infty \left\{(x_{-1},x_0, y_{-1},
y_0)\in \left(\mathbb{R}-\left\{0\right\}\right) :A_n=0 \,\text{or} \,B_n=0\right\},
\end{equation*}
where $$A_n=J_{n+1}y_0x_{-1}+(J_{n+2}-aJ_{n+1})x_{-1}+cJ_{n},\,\,B_n=J_{n+1}x_0y_{-1}+(J_{n+2}-aJ_{n+1})y_{-1}+cJ_{n}.$$
\end{theorem}
\begin{proof}
Putting
\begin{equation} \label{changing_variable_x(n)_y(n)_by_u(n)_v(n)}
x_{n}=\dfrac{u_{n}}{v_{n-1}},\quad y_{n}=\dfrac{v_{n}}{u_{n-1}},\,n=-1,0,1,...
\end{equation}
we get the following linear third order system of  difference
equations
\begin{equation}\label{sys_linear}
u_{n+1}=av_{n}+bu_{n-1}+cv_{n-2},\quad
v_{n+1}=au_{n}+bv_{n-1}+cu_{n-2},\quad
 n=0,1,...,
\end{equation}
where the initial values $u_{-2},u_{-1},u_0,v_{-2},v_{-1},v_0$ are nonzero real numbers.\\
From\eqref{sys_linear} we have for $n=0,1,...,$
\begin{equation*}
\begin{cases}
u_{n+1}+v_{n+1}=a(v_{n}+u_{n})+b(u_{n-1}+v_{n-1})+c(v_{n-2}+u_{n-2}),\\
u_{n+1}-v_{n+1}=a(v_{n}-u_{n})+b(u_{n-1}-v_{n-1})+c(v_{n-2}-u_{n-2}).
\end{cases}
\end{equation*}
Putting again
\begin{equation}\label{u,v}
R_n=u_n+v_n, \quad S_n=u_n-v_n,\,n=-2,-1,0,...,
\end{equation}\label{R,S}
we obtain two homogenous linear difference equations of third order:
\begin{equation*}
R_{n+1}=aR_{n}+bR_{n-1}+cR_{n-2},\,n=0,1,\cdots{},
\end{equation*}
and
\begin{equation} \label{S(n)_equation}
S_{n+1}=-aS_{n}+bS_{n-1}-cS_{n-2},\,n=0,1,\cdots{}.
\end{equation}

Using \eqref{u,v}, we get for $n=-2,-1,0,...,$
\begin{equation*}
u_n= \dfrac{1}{2}(R_n+S_n), \; v_n= \dfrac{1}{2}(R_n-S_n).
\end{equation*}
From Lemma \ref{lemma_R(n),S(n)_by_J(n)} we obtain,

\begin{eqnarray}
\begin{cases} \label{solution_form_of_u(n)}
u_{2n-1}=\dfrac{1}{2}\left[ cJ_{2n-1}(R_{-2}-S_{-2})+(J_{2n+1}-aJ_{2n})(R_{-1}+S_{-1})+J_{2n}(R_{0}-S_{0}) \right], n=1,2,\cdots,\\
u_{2n}=\dfrac{1}{2}\left[ cJ_{2n}(R_{-2}+S_{-2})+(J_{2n+2}-aJ_{2n+1})(R_{-1}-S_{-1})+J_{2n+1}(R_{0}+S_{0}) \right], n=0,1,\cdots,
\end{cases}
\end{eqnarray}
\begin{eqnarray}
\begin{cases} \label{solution_form_of_v(n)}
v_{2n-1}= \dfrac{1}{2}\left[ cJ_{2n-1}(R_{-2}+S_{-2})+(J_{2n+1}-aJ_{2n})(R_{-1}-S_{-1})+J_{2n}(R_{0}+S_{0}) \right], n=1,2,\cdots,\\
v_{2n}= \dfrac{1}{2}\left[ cJ_{2n}(R_{-2}-S_{-2})+(J_{2n+2}-aJ_{2n+1})(R_{-1}+S_{-1})+J_{2n+1}(R_{0}-S_{0}) \right], n=0,1,\cdots,
\end{cases}
\end{eqnarray}
Substituting \eqref{solution_form_of_u(n)} and \eqref{solution_form_of_v(n)} in \eqref{changing_variable_x(n)_y(n)_by_u(n)_v(n)}, we get for $n=0,1,...,$
\begin{eqnarray}
\begin{cases}
x_{2n+2}=\dfrac{ cJ_{2n+2}(R_{-2}+S_{-2})+(J_{2n+4}-aJ_{2n+3})(R_{-1}-S_{-1})+J_{2n+3}(R_{0}+S_{0})}{ cJ_{2n+1}(R_{-2}+S_{-2})+(J_{2n+3}-aJ_{2n+2})(R_{-1}-S_{-1})+J_{2n+2}(R_{0}+S_{0})},\\
x_{2n+1}=\dfrac{cJ_{2n+1}(R_{-2}-S_{-2})+(J_{2n+3}-aJ_{2n+2})(R_{-1}+S_{-1})+J_{2n+2}(R_{0}-S_{0})}{ cJ_{2n}(R_{-2}-S_{-2})+(J_{2n+2}-aJ_{2n+1})(R_{-1}+S_{-1})+J_{2n+1}(R_{0}-S_{0})},
\end{cases}
\end{eqnarray}
\begin{eqnarray}
\begin{cases}
y_{2n+2}=\dfrac{ cJ_{2n+2}(R_{-2}-S_{-2})+(J_{2n+4}-aJ_{2n+3})(R_{-1}+S_{-1})+J_{2n+3}(R_{0}-S_{0})}
{cJ_{2n+1}(R_{-2}-S_{-2})+(J_{2n+3}-aJ_{2n+2})(R_{-1}+S_{-1})+J_{2n+2}(R_{0}-S_{0})},\\
y_{2n+1}=\dfrac{ cJ_{2n+1}(R_{-2}+S_{-2})+(J_{2n+3}-aJ_{2n+2})(R_{-1}-S_{-1})+J_{2n+2}(R_{0}+S_{0})}
{cJ_{2n}(R_{-2}+S_{-2})+(J_{2n+2}-aJ_{2n+1})(R_{-1}-S_{-1})+J_{2n+1}(R_{0}+S_{0})}.
\end{cases}
\end{eqnarray}
Then,
\begin{eqnarray}\label{x,,}
\begin{cases}
x_{2n+2}=\dfrac{cJ_{2n+2}+(J_{2n+4}-aJ_{2n+3})\dfrac{R_{-1}-S_{-1}}{R_{-2}+S_{-2}}+J_{2n+3}\dfrac{R_{0}+S_{0}}{R_{-2}+S_{-2}}}{ cJ_{2n+1}+(J_{2n+3}-aJ_{2n+2})\dfrac{R_{-1}-S_{-1}}{R_{-2}+S_{-2}}+J_{2n+2}\dfrac{R_{0}+S_{0}}{R_{-2}+S_{-2}}},\\
x_{2n+1}=\dfrac{cJ_{2n+1}+(J_{2n+3}-aJ_{2n+2})\dfrac{R_{-1}+S_{-1}}{R_{-2}-S_{-2}}+J_{2n+2}\dfrac{R_{0}-S_{0}}{R_{-2}-S_{-2}}}{ cJ_{2n}+(J_{2n+2}-aJ_{2n+1})\dfrac{R_{-1}+S_{-1}}{R_{-2}-S_{-2}}+J_{2n+1}\dfrac{R_{0}-S_{0}}{R_{-2}-S_{-2}}},
\end{cases}
\end{eqnarray}
\begin{eqnarray}\label{y,,}
\begin{cases}
y_{2n+2}=\dfrac{ cJ_{2n+2}+(J_{2n+4}-aJ_{2n+3})\dfrac{R_{-1}+S_{-1}}{R_{-2}-S_{-2}}+J_{2n+3}\dfrac{R_{0}-S_{0}}{R_{-2}-S_{-2}}}{cJ_{2n+1}+(J_{2n+3}-aJ_{2n+2})
\dfrac{R_{-1}+S_{-1}}{R_{-2}-S_{-2}}+J_{2n+2}\dfrac{R_{0}-S_{0}}{R_{-2}-S_{-2}}},\\
y_{2n+1}=\dfrac{ cJ_{2n+1}+(J_{2n+3}-aJ_{2n+2})\dfrac{R_{-1}-S_{-1}}{R_{-2}+S_{-2}}+J_{2n+2}\dfrac{ R_{0}+S_{0}}{R_{-2}+S_{-2}}}{cJ_{2n}+(J_{2n+2}-aJ_{2n+1})\dfrac{ R_{-1}-S_{-1}}{R_{-2}+S_{-2}}+J_{2n+1}\dfrac{R_{0}+S_{0}}{R_{-2}+S_{-2}}}.
\end{cases}
\end{eqnarray}
We have
\begin{equation} \label{xrs}
  x_{-1}=\dfrac{u_{-1}}{v_{-2}}=\dfrac{R_{-1}+S_{-1}}{R_{-2}-S_{-2}},\,x_0=\dfrac{u_0}{v_{-1}}=\dfrac{R_{0}+S_{0}}{R_{-1}-S_{-1}},
  \end{equation}
  \begin{equation} \label{yrs}
  y_{-1}=\dfrac{v_{-1}}{u_{-2}}=\dfrac{R_{-1}-S_{-1}}{R_{-2}+S_{-2}},\, y_0=\dfrac{v_0}{u_{-1}}=\dfrac{R_{0}-S_{0}}{R_{-1}+S_{-1}}
\end{equation}
From \eqref{xrs}, \eqref{yrs} it follows that,

\begin{equation}\label{xyrs}
\begin{cases}
\dfrac{R_{0}-S_{0}}{R_{-2}-S_{-2}}=\dfrac{R_{-1}+S_{-1}}{R_{-2}-S_{-2}} \times \dfrac{R_{0}-S_{0}}{R_{-1}+S_{-1}}= x_{-1}y_0 \\
\dfrac{R_{0}+S_{0}}{R_{-2}+S_{-2}}= \dfrac{R_{0}+S_{0}}{R_{-1}-S_{-1}} \times \dfrac{R_{-1}-S_{-1}}{R_{-2}+S_{-2}}=x_0y_{-1}
\end{cases}
\end{equation}
Using \eqref{x,,}, \eqref{y,,}, \eqref{xrs}, \eqref{yrs} and \eqref{xyrs}, we obtain the closed form of the solutions of \eqref{system-generalized}, that is for $n=0,1,...,$ we have
\begin{eqnarray*}
\begin{cases}
x_{2n+2}=\dfrac{cJ_{2n+2}+(J_{2n+4}-aJ_{2n+3})y_{-1}+J_{2n+3}x_{0}y_{-1}}{ cJ_{2n+1}+(J_{2n+3}-aJ_{2n+2})y_{-1}+J_{2n+2}x_{0}y_{-1}},\\
x_{2n+1}=\dfrac{cJ_{2n+1}+(J_{2n+3}-aJ_{2n+2})x_{-1}+J_{2n+2}x_{-1}y_{0}}{ cJ_{2n}+(J_{2n+2}-aJ_{2n+1})x_{-1}+J_{2n+1}x_{-1}y_{0}},
\end{cases}
\end{eqnarray*}
\begin{eqnarray*}
\begin{cases}
y_{2n+2}=\dfrac{ cJ_{2n+2}+(J_{2n+4}-aJ_{2n+3})x_{-1}+J_{2n+3}x_{-1}y_{0}}{cJ_{2n+1}+(J_{2n+3}-aJ_{2n+2})
x_{-1}+J_{2n+2}x_{-1}y_{0}},\\
y_{2n+1}=\dfrac{ cJ_{2n+1}+(J_{2n+3}-aJ_{2n+2})y_{-1}+J_{2n+2}x_{0}y_{-1}}{cJ_{2n}+(J_{2n+2}-aJ_{2n+1})y_{-1}+J_{2n+1}x_{0}y_{-1}}.
\end{cases}
\end{eqnarray*}
\end{proof}

\begin{remark}
Writing system \eqref{system-generalized} in the form

\begin{equation*}
    \left\{
    \begin{array}{ll}
        x_{n+1}= & f\left( x_{n},x_{n-1},y_{n},y_{n-1}\right)=\frac{ay_{n}x_{n-1}+bx_{n-1}+c}{y_{n}x_{n-1}}  \\
        y_{n+1}= & g\left( x_{n},x_{n-1},y_{n},y_{n-1}\right)=\frac{ax_{n}y_{n-1}+by_{n-1}+c}{x_{n}y_{n-1}}.
    \end{array}
    \right.
\end{equation*}
So it follows that points $(\alpha,\alpha)$, $(\beta,\beta)$  and $(\gamma,\gamma)$  are solutions of the of system
\begin{equation*}
\begin{cases}
\bar{x}=\dfrac{a\bar{y}\bar{x}+b\bar{x}+c}{\bar{y}\bar{x}},\\
\bar{y}=\dfrac{a\bar{x}\bar{y}+b\bar{y}+c}{\bar{x}\bar{y}}
\end{cases}
\end{equation*}
where $\alpha$, $\beta$ and $\gamma$ are the roots of \eqref{eq_charac_of_GTS}.
\end{remark}

\begin{theorem} \label{thm_global_stability_of_system_generalized}
Under the same conditions in Lemma \ref{limits_of_J(n)}, for every well defined solution of system \eqref{system-generalized}, we have $$\lim_{n\rightarrow+\infty}x_{2n+1}=\lim_{n\rightarrow+\infty}x_{2n+2}=\lim_{n\rightarrow+\infty}y_{2n+1}=\lim_{n\rightarrow+\infty}y_{2n+2}=\alpha.$$
\end{theorem}

\begin{proof}
We have

\begin{eqnarray*}
\lim\limits_{n\rightarrow\infty} x_{2n+1} &=& \lim\limits_{n\rightarrow\infty} \dfrac{cJ_{2n+1}+(J_{2n+3}-aJ_{2n+2})x_{-1}+J_{2n+2}y_0x_{-1}}{ cJ_{2n}+(J_{2n+2}-aJ_{2n+1})x_{-1}+J_{2n+1}y_0x_{-1}}\\
&=& \lim\limits_{n\rightarrow\infty} \dfrac{cJ_{2n+1}+(J_{2n+3}-aJ_{2n+2})x_{-1}+J_{2n+2}y_0x_{-1}}{ cJ_{2n}+(J_{2n+2}-aJ_{2n+1})x_{-1}+J_{2n+1}y_0x_{-1}}\\
&=& \lim\limits_{n\rightarrow\infty} \dfrac{c\dfrac{J_{2n+1}}{J_{2n}}+\left(\dfrac{J_{2n+3}}{J_{2n+2}} \times \dfrac{J_{2n+2}}{J_{2n+1}}\times\dfrac{J_{2n+1}}{J_{2n}} -a \dfrac{J_{2n+2}}{J_{2n+1}} \times \dfrac{J_{2n+1}}{J_{2n}}\right) x_{-1}+\dfrac{J_{2n+2}}{J_{2n+1}} \times \dfrac{J_{2n+1}}{J_{2n}} y_0x_{-1}}{ c\dfrac{J_{2n}}{J_{2n}}+\left( \dfrac{J_{2n+2}}{J_{2n+1}} \times\dfrac{J_{2n+1}}{J_{2n}} -a\dfrac{J_{2n+1}}{J_{2n}} \right) x_{-1}+\dfrac{J_{2n+1}}{J_{2n}}y_0x_{-1}}\\
&=& \dfrac{c\alpha+(\alpha^3-a\alpha^2)x_{-1}+\alpha^2 y_0x_{-1}}{c+(\alpha^2-a\alpha)x_{-1}+\alpha y_0x_{-1}} \\
&=& \alpha \; .
\end{eqnarray*}
In the same way we show that
$$\lim\limits_{n\rightarrow\infty} x_{2n+2}= \lim\limits_{n\rightarrow\infty} y_{2n+1}=\lim\limits_{n\rightarrow\infty} y_{2n+1}=\alpha.$$
\end{proof}
\section{Particular cases} Here we are interested in some particular cases of system \eqref{system-generalized}. Some of these particular cases was been the subject of some recent papers.
\subsection{The solutions of the equation $x_{n+1}=\dfrac{ax_nx_{n-1}+bx_{n-1}+c}{x_nx_{n-1}}$}

If we choose  $y_{-1}=x_{-1}$ and $y_0=x_0$, then system \eqref{system-generalized} is reduced to the equation
\begin{equation}\label{equation-generalized}
   x_{n+1}=\dfrac{ax_nx_{n-1}+bx_{n-1}+c}{x_nx_{n-1}}, \; \; n\in \mathbb{N}_0.
\end{equation}

The following results are respectively  direct consequences of Theorem \ref{solution_form_of_x(n)_and_y(n)} and Theorem \ref{thm_global_stability_of_system_generalized}.

\begin{corollary} \label{solution_form_of_equation_generalized}
Let $\{ x_n \}_{n\geq -1}$ be a well defined solution of the equation \eqref{equation-generalized}. Then for $n=0,1,\ldots,$ we have

\begin{equation*}
x_{2n+1}=\dfrac{cJ_{2n+1}+(J_{2n+3}-aJ_{2n+2})x_{-1}+J_{2n+2}x_{-1}x_{0}}{ cJ_{2n}+(J_{2n+2}-aJ_{2n+1})x_{-1}+J_{2n+1}x_{-1}x_{0}},
\end{equation*}

\begin{equation*}
x_{2n+2}=\dfrac{cJ_{2n+2}+(J_{2n+4}-aJ_{2n+3})x_{-1}+J_{2n+3}x_{0}x_{-1}}{ cJ_{2n+1}+(J_{2n+3}-aJ_{2n+2})x_{-1}+J_{2n+2}x_{0}x_{-1}}.
\end{equation*}
\end{corollary}

\begin{corollary}
Under the same conditions in Lemma \ref{limits_of_J(n)}, for every well defined solution of equation  \eqref{equation-generalized}, we have $$\lim_{n\rightarrow+\infty}x_{2n+1}=\lim_{n\rightarrow+\infty}x_{2n+2}=\alpha.$$
\end{corollary}
The equation \eqref{equation-generalized} was been studied by Azizi in \cite{azizi} and Stevic in \cite{stevicejqtde2018}.

\subsection{The solutions of the system $x_{n+1}=\dfrac{y_nx_{n-1}+x_{n-1}+1}{y_nx_{n-1}},\; y_{n+1}=\dfrac{x_ny_{n-1}+y_{n-1}+1}{x_ny_{n-1}}$}
Consider the system
\begin{equation}\label{system:1}
   x_{n+1}=\dfrac{y_nx_{n-1}+x_{n-1}+1}{y_nx_{n-1}},\; y_{n+1}=\dfrac{x_ny_{n-1}+y_{n-1}+1}{x_ny_{n-1}}
\; \; n\in \mathbb{N}_0.
\end{equation}
Clearly the system \eqref{system:1} is particular of the system \eqref{system-generalized} with $a=b=c=1$. In this case the sequence  $\{J_n \}$ is the famous classical sequence of Tribonacci numbers $\{T_n \}$, that is
\begin{equation*}
T_{n+3}=T_{n+2}+T_{n+1}+T_{n},\quad n\in \mathbb{N},  \; \text{ where } \; T_0=0, \;T_1=1 \; \text{and} \; T_2=1,
\end{equation*}
and we have
\begin{equation*}
T_{n}=\dfrac{\alpha^{n+1}}{(\beta-\alpha)(\gamma-\alpha)}-\dfrac{\beta^{n+1}}{(\beta-\alpha)(\gamma-\beta)}+\dfrac{\gamma^{n+1}}{(\gamma-\alpha)(\gamma-\beta)}, \qquad  n=0,1,...,
\end{equation*}
with
$$\alpha= \dfrac{1+\sqrt[3]{19+3\sqrt{33}}+\sqrt[3]{19-3\sqrt{33}}}{3},\,\beta= \dfrac{1+\omega\sqrt[3]{19+3\sqrt{33}}+\omega^2\sqrt[3]{19-3\sqrt{33}}}{3},$$ $$\gamma=\dfrac{1+\omega^2\sqrt[3]{19+3\sqrt{33}}+\omega\sqrt[3]{19-3\sqrt{33}}}{3},\,\omega=\dfrac{-1+i\sqrt{3}}{2}.$$

Numerically we have $\alpha=1.839286755$ and the two complex conjugate are $$-0.4196433777+0.6062907300 i,\,-0.4196433777-0.6062907300 i$$ with $i^{2}=-1$.

The following results follows respectively from Theorem \ref{solution_form_of_x(n)_and_y(n)} and Theorem \ref{thm_global_stability_of_system_generalized}.

\begin{corollary}
Let $\{x_n, y_n\}_{n\geq -1}$ be a well defined solution of
\eqref{system:1}. Then, for $n=0,1,2,3,\ldots,$ we have

\begin{equation*}
x_{2n+1}=\dfrac{cT_{2n+1}+(T_{2n+3}-aT_{2n+2})x_{-1}+T_{2n+2}x_{-1}y_{0}}{ cT_{2n}+(T_{2n+2}-aT_{2n+1})x_{-1}+T_{2n+1}x_{-1}y_{0}},
\end{equation*}

\begin{equation*}
x_{2n+2}=\dfrac{cT_{2n+2}+(T_{2n+4}-aT_{2n+3})y_{-1}+T_{2n+3}x_{0}y_{-1}}{ cT_{2n+1}+(T_{2n+3}-aT_{2n+2})y_{-1}+T_{2n+2}x_{0}y_{-1}},
\end{equation*}

\begin{equation*}
y_{2n+1}=\dfrac{ cT_{2n+1}+(T_{2n+3}-aT_{2n+2})y_{-1}+T_{2n+2}x_{0}y_{-1}}{cT_{2n}+(T_{2n+2}-aT_{2n+1})y_{-1}+T_{2n+1}x_{0}y_{-1}},
\end{equation*}

\begin{equation*}
y_{2n+2}=\dfrac{ cT_{2n+2}+(T_{2n+4}-aT_{2n+3})x_{-1}+T_{2n+3}x_{-1}y_{0}}{cT_{2n+1}+(T_{2n+3}-aT_{2n+2})x_{-1}+T_{2n+2}x_{-1}y_{0}}
\end{equation*}
\end{corollary}

\begin{corollary}
For every well defined solution of system \eqref{system-generalized}, we have $$\lim_{n\rightarrow+\infty}x_{2n+1}=\lim_{n\rightarrow+\infty}x_{2n+2}=\lim_{n\rightarrow+\infty}y_{2n+1}=\lim_{n\rightarrow+\infty}y_{2n+2}=\alpha.$$
\end{corollary}\label{cort}
From the equation

\begin{equation}\label{equation_tribonacci}
   x_{n+1}=\dfrac{x_nx_{n-1}+x_{n-1}+1}{x_nx_{n-1}}, \; \; n\in \mathbb{N}_0.
\end{equation}
we have the following results.
\begin{corollary} \label{solution_form_of_equation_tribonacci}
Let $\{ x_n \}_{n\geq -1}$ be a well defined solution of the equation \eqref{equation_tribonacci}. Then for $n=0,1,\ldots,$ we have

\begin{equation*}
x_{2n+1}=\dfrac{T_{2n+1}+(T_{2n+3}-T_{2n+2})x_{-1}+T_{2n+2}x_{-1}x_{0}}{ T_{2n}+(T_{2n+2}-T_{2n+1})x_{-1}+T_{2n+1}x_{-1}x_{0}},
\end{equation*}

\begin{equation*}
x_{2n+2}=\dfrac{T_{2n+2}+(T_{2n+4}-T_{2n+3})x_{-1}+T_{2n+3}x_{0}x_{-1}}{ T_{2n+1}+(T_{2n+3}-T_{2n+2})x_{-1}+T_{2n+2}x_{0}x_{-1}}.
\end{equation*}
\end{corollary}

\begin{corollary}\label{limittribonacci}
Under the same conditions in Lemma \ref{limits_of_J(n)}, for every well defined solution of the equation \eqref{equation_tribonacci}, we have $$\lim_{n\rightarrow+\infty}x_{2n+1}=\lim_{n\rightarrow+\infty}x_{2n+2}=\alpha.$$
\end{corollary}.

Let $I=\left(0,+\infty\right)$,  $J=\left(0,+\infty\right)$ and choosing $x_{-1}$, $x_{0}$, $y_{-1}$ and $y_0$ $\in \left(0,+\infty\right)$. Then clearly the system
\begin{equation*}
   \overline{x}=f(\overline{x},\overline{y})=\dfrac{\overline{x}\overline{y}+\overline{x}+1}{\overline{x}\overline{y}},\; \overline{y}=g(\overline{x},\overline{y})=\dfrac{\overline{x}\overline{y}+\overline{y}+1}{\overline{x}\overline{y}}
\end{equation*}
has a unique solution $(\alpha,\alpha)$ $\in$ $I \times J$, that is $(\alpha,\alpha)$ is the unique equilibrium point (fixed point) of our system \begin{equation*}
   x_{n+1}=f(x_{n},x_{n-1},y_{n},y_{n-1})=\dfrac{y_nx_{n-1}+x_{n-1}+1}{y_nx_{n-1}},
   \end{equation*}
 \begin{equation*}
   y_{n+1}=g(x_{n},x_{n-1},y_{n},y_{n-1})=\dfrac{x_ny_{n-1}+y_{n-1}+1}{x_ny_{n-1}}.
\end{equation*}

Clearly the functions

\begin{equation*}
f : I^2 \times J^2 \longrightarrow I \qquad \text{and} \qquad g :I^2 \times J^2 \longrightarrow I
\end{equation*}
defined by
\begin{equation*}
f(u_0; u_1; v_0; v_1) = \dfrac{v_0 u_1+u_1+1}{v_0u_1} \qquad \text{and} \qquad
g(u_0; u_1; v_0; v_1) = \dfrac{u_0v_1+v_1+1}{u_0v_1}
\end{equation*}
 are continuously differentiable.

In the following result we prove that the unique equilibrium point $(\alpha,\alpha)$ of \eqref{system:1} is locally asymptotically stable..
\begin{theorem} \label{thm-loca-asym-stabl-system}
The equilibrium point $(\alpha,\alpha)$ is locally asymptotically stable.
\end{theorem}\label{thlas}
\begin{proof}
The Jacobian matrix associated to the system \eqref{system:1} around the equilibrium point $(\alpha,\alpha)$, is given by
\begin{equation*}
A=\begin{pmatrix}
0 & -\dfrac{1}{\alpha^3} & -\dfrac{\alpha +1}{\alpha^3} & 0 \\
1 & 0 & 0 & 0 & 0 \\
-\dfrac{\alpha +1}{\alpha^3} & 0  & 0 & -\dfrac{1}{\alpha^3} \\
0 & 0 & 0 & 1 & 0
\end{pmatrix}
\end{equation*}
Then, the characteristic polynomial of $A$ is
$$P(\lambda)=\lambda^4+\dfrac{(2\alpha^3-\alpha^2-2\alpha-1)}{\alpha^6} \lambda^2+\dfrac{1}{\alpha^6} $$
and the roots of $P(\lambda)$ are
$$\lambda_1 = \dfrac{1}{2} \times \dfrac{1+\alpha+\sqrt{-4\alpha^3+\alpha^2+2\alpha+1}}{\alpha^3},\,\lambda_2 = -\dfrac{1}{2} \times \dfrac{1+\alpha+\sqrt{-4\alpha^3+\alpha^2+2\alpha+1}}{\alpha^3},$$
$$\lambda_3 = \dfrac{1}{2} \times \dfrac{-1-\alpha+\sqrt{-4\alpha^3+\alpha^2+2\alpha+1}}{\alpha^3},\,\lambda_4 = -\dfrac{1}{2} \times \dfrac{-1-\alpha+\sqrt{-4\alpha^3+\alpha^2+2\alpha+1}}{\alpha^3}.$$

We have $\vert \lambda_i \vert<1 , \; i=1,2,3,4$, so the equilibrium point $(\alpha,\alpha)$ is locally asymptotically stable.
\end{proof}
The following result is a direct consequence of Theorem \ref{thlas} and Corollary \ref{cort}.
\begin{theorem}
The equilibrium point $(\alpha,\alpha)$ is globally asymptotically stable .
\end{theorem}

Let $I=\left(0,+\infty\right)$ and choosing $x_{-1}$, $x_{0}$ $\in \left(0,+\infty\right)$. Writing the equation \eqref{equation_tribonacci} as

\begin{equation}\label{eqh}
   x_{n+1}=h(x_{n},x_{n-1})=\dfrac{x_nx_{n-1}+x_{n-1}+1}{x_nx_{n-1}}
\end{equation}
where \begin{equation*}
h : I^2 \longrightarrow I
\end{equation*}
is defined by
\begin{equation*}
h(u_0; u_1) = \dfrac{u_0 u_1+u_1+1}{u_0u_1}.
\end{equation*}
 The function $h$ is continuously differentiable. The equation $\overline{x}=h(\overline{x},\overline{x})$ has the unique solution  $\overline{x}=\alpha$ in $ \left(0,+\infty\right)$. The linear equation associated to the equation \eqref{eqh} about the equilibrium point $\overline{x}=\alpha$ is given by
 $$y_{n+1}=\frac{\partial h}{\partial u_{0}}\left(\alpha,\alpha\right)y_{n}+\frac{\partial h}{\partial u_{1}}\left(\alpha,\alpha\right)y_{n-1},$$ the last equation has as characteristic polynomial
 \begin{equation*}
Q(\lambda)=\lambda^{2}-\frac{\partial h}{\partial u_{0}}\left(\alpha,\alpha\right)\lambda-\frac{\partial h}{\partial u_{1}}\left(\alpha,\alpha\right).
\end{equation*}

 In the following result we show that the unique equilibrium point $\overline{x}=\alpha$ is globally stable.

\begin{theorem}
The equilibrium point $\overline{x}=\alpha$ is globally stable.
\end{theorem}
\begin{proof}
The linear equation associated to \eqref{equation_tribonacci} about the equilibrium point $\overline{x}=\alpha$ is
$$y_{n+1}=-\dfrac{\alpha +1}{\alpha^3} y_n - \dfrac{1}{\alpha^3} y_{n-1}$$
and the characteristic polynomial is
$$Q(\lambda)=\lambda^{2}+\left(\dfrac{\alpha +1}{\alpha^3}\right) \lambda + \left(\dfrac{1}{\alpha^3}\right).$$
We have $$\vert \dfrac{\alpha +1}{\alpha^3} \vert + \vert \dfrac{1}{\alpha^3} \vert <1.$$
So, by Rouch\'e's theorem the roots of the characteristic polynomial $Q(\lambda)$ lie in the unit disk. Then the equilibrium point $\overline{x}=\alpha$ is locally asymptotically stable. Now, from this and  Corollary \ref{limittribonacci} the result holds.
\end{proof}

\subsection{The system $x_{n+1}=\dfrac{bx_{n-1}+c}{y_nx_{n-1}},\; y_{n+1}=\dfrac{by_{n-1}+c}{x_ny_{n-1}}$}

When $a=0$, the system \eqref{system-generalized} takes the form

\begin{equation}\label{system_generalized_padovan}
   x_{n+1}=\dfrac{bx_{n-1}+c}{y_nx_{n-1}},\; y_{n+1}=\dfrac{by_{n-1}+c}{x_ny_{n-1}}
\; \; n\in \mathbb{N}_0.
\end{equation}
 From Theorem \eqref{solution_form_of_x(n)_and_y(n)}, we get the following result.

\begin{corollary}\label{a=0}
Let $\{x_n, y_n\}_{n\geq -1}$ be a well defined solution of
\eqref{system_generalized_padovan}. Then, for $n=0,1,\ldots,$ we have

\begin{equation*}
x_{2n+1}=\dfrac{cP_{2n+1}+P_{2n+3}x_{-1}+P_{2n+2}x_{-1}y_{0}}{ cP_{2n}+P_{2n+2}x_{-1}+P_{2n+1}x_{-1}y_{0}},
\end{equation*}

\begin{equation*}
x_{2n+2}=\dfrac{cP_{2n+2}+P_{2n+4}y_{-1}+P_{2n+3}x_{0}y_{-1}}{ cP_{2n+1}+P_{2n+3}y_{-1}+P_{2n+2}x_{0}y_{-1}},
\end{equation*}

\begin{equation*}
y_{2n+1}=\dfrac{ cP_{2n+1}+P_{2n+3}y_{-1}+P_{2n+2}x_{0}y_{-1}}{cP_{2n}+P_{2n+2}y_{-1}+P_{2n+1}x_{0}y_{-1}},
\end{equation*}

\begin{equation*}
y_{2n+2}=\dfrac{ cP_{2n+2}+P_{2n+4}x_{-1}+P_{2n+3}x_{-1}y_{0}}{cP_{2n+1}+P_{2n+3}x_{-1}+P_{2n+2}x_{-1}y_{0}}.
\end{equation*}
\end{corollary}
Here we have write $\{P_{n}\}_{n}$ instead of $\{J_{n}\}_{n}$, as in this case $\{J_{n}\}_{n}$ takes the form of a generalized (Padovan) sequence, that is
$$P_{n+3}=bP_{n+1}+cP_{n},\quad n\in \mathbb{N},$$  with special values $P_0=0$, $P_1=1$ and $P_2=0$.
The system \eqref{system_generalized_padovan} was been investigated by  Halim et al. in \cite{Halim Rabago:2018} and by Yazlik et al. in  \cite{Yazlik Tollu Taskara:2013} with  $b=1$ and $c=\pm1$. The one dimensional version of system \eqref{system_generalized_padovan}, that is the equation
\begin{equation} \label{eq_generalized_padovan}
x_{n+1}=\dfrac{bx_{n-1}+c}{x_nx_{n-1}}, \; \; n\in \mathbb{N}_0.
\end{equation}
was been also investigated by Halim et al. in \cite{Halim Rabago:2018}. Form Corollary \eqref{a=0}, we get that the well defined solutions of equation \eqref{eq_generalized_padovan} are given for $n=0,1,...,$ by
\begin{equation*}
x_{2n+1}=\dfrac{cP_{2n+1}+P_{2n+3}x_{-1}+P_{2n+2}x_{-1}x_{0}}{ cP_{2n}+P_{2n+2}x_{-1}+P_{2n+1}x_{-1}x_{0}},
\end{equation*}

\begin{equation*}
x_{2n+2}=\dfrac{cP_{2n+2}+P_{2n+4}x_{-1}+P_{2n+3}x_{0}x_{-1}}{ cP_{2n+1}+P_{2n+3}x_{-1}+P_{2n+2}x_{0}x_{-1}}.
\end{equation*}
In \cite{Yazlik Tollu Taskara:2013} and \cite{Halim Rabago:2018} we can find additional results on the stability of some equilibrium points.
\begin{remark}
If $c=0$, The system \eqref{system-generalized} become

\begin{equation}\label{system_generalized_fibonaci}
   x_{n+1}=\dfrac{ay_{n}+b}{y_{n}},\; y_{n+1}=\dfrac{ax_{n}+b}{x_{n}},\, n\in \mathbb{N}_0.
\end{equation}
We note that if also $b=0$, then the solutions of the system \eqref{system_generalized_fibonaci} are given by
$$\left\{\left(x_{0},y_{0}\right),\left(a,a\right),\left(a,a\right),...,\right\}.$$
The  system \eqref{system_generalized_fibonaci} is a particular case of the more general system

\begin{equation}\label{ss}
 x_{n+1}=\dfrac{ay_{n}+b}{cy_{n}+d},\; y_{n+1}=\dfrac{\alpha x_{n}+\beta}{\gamma x_{n}+\lambda},\; n\in \mathbb{N}_0
\end{equation}
which was been completely solved by Stevic in \cite{stevicejqtde}. So, we refer to this paper for the readers interested in the form of the solutions of the system \eqref{ss} and its particular case system \eqref{system_generalized_fibonaci}. As it was proved in \cite{stevicejqtde}, the solutions are expressed using the terms of a corresponding generalized Fibonacci sequence. Noting that the papers \cite{matsunaga},\cite{Tollu Yazlik Taskara:2014(1)}, \cite{Tollu Yazlik Taskara:2013} deals also with particular cases of the system \eqref{ss} or its one dimensional version.
\end{remark}


%
%
\bibliography{mmnsample}
\bibliographystyle{mmn}

\end{document}